\documentclass[reqno,12 pt]{amsart}

\usepackage{amsfonts,amscd}
\usepackage{amsthm}
\usepackage{amssymb}
\usepackage{enumitem}
\usepackage{latexsym}
\usepackage{wasysym}
\usepackage{stmaryrd}
\usepackage{color}
\usepackage{url}
\usepackage{cite}
\usepackage{amsmath}
\usepackage{verbatim}
\usepackage{graphicx}

\usepackage{xcolor}
\usepackage{multirow}
\usepackage{amssymb}

\usepackage{algorithm}
\usepackage{algorithmicx}
\usepackage{algpseudocode}
\usepackage{epsf}
\usepackage{geometry}
\usepackage{hyperref}
\definecolor{dark-blue}{rgb}{0,0,0.6}
\definecolor{Purple}{rgb}{0.2,0,0.25}
\hypersetup{breaklinks=true,
pagecolor=white,
colorlinks=true,
citecolor=dark-blue,
filecolor=black,%
linkcolor=Purple,%
urlcolor=blue
}

\newtheorem{theorem}{Theorem}[section]

\newtheorem{lemma}[theorem]{Lemma}
\theoremstyle{definition}
\newtheorem{example}[theorem]{Example}%
\newtheorem{remark}[theorem]{Remark}%
\newtheorem{alg}[theorem]{Algorithm}

\numberwithin{equation}{section}

\date{June 12, 2022}


\begin{document}

\title[A Generalized Block-Iterative Projection method 
for Cutters]{A Generalized Block-Iterative Projection Method for the Common Fixed Point Problem Induced by Cutters}

\author{Yair Censor}
\address{ Department of Mathematics, University of Haifa, Mt. Carmel, 3498838 Haifa, Israel.} 
\email{(Yair Censor) yair@math.haifa.ac.il}
\author{Daniel Reem}
\address{The Center for Mathematics and Scientific Computation (CMSC), University of Haifa, Mt. Carmel, Haifa, 3498838,  Israel.} 
\email{(Daniel Reem) dream@math.haifa.ac.il}
\author{Maroun Zaknoon}
\address{Department of Mathematics, The Arab Academic College for Education, 22 HaHashmal Street, Haifa, 32623, Israel}
\email{(Maroun Zaknoon) zaknoon@arabcol.ac.il}

\maketitle

\begin{abstract}
The block-iterative projections (BIP) method of Aharoni and Censor [Block-iterative projection methods for parallel computation of solutions to convex feasibility problems, Linear Algebra and its Applications 120, (1989), 165–175] is an iterative process for finding asymptotically a point in the nonempty intersection of a family of closed convex subsets. It employs orthogonal projections onto the individual subsets in an algorithmic regime that uses “blocks” of operators and has great flexibility in constructing specific algorithms from it. We extend this algorithmic scheme to handle a family of continuous cutter operators and to find a common fixed point of them. Since the family of continuous cutters includes several important specific operators, our generalized scheme, which ensures global convergence and retains the flexibility of BIP, can handle, in particular, metric (orthogonal) projectors and continuous subgradient projections, which are very important in applications. We also allow a certain kind of adaptive perturbations to be included, and along the way we derive a perturbed Fej\'er monotonicity lemma which is of independent interest.
\end{abstract}

\section{Introduction}\label{sec:Intro}
\subsection{Background and contributions}\label{subsec:Background}

Given a finite family of $m\in\mathbb{N}$ nonempty closed convex
subsets $C_{1},C_{2},\ldots,C_{m}$ of the $n$-dimensional Euclidean
space $\mathbb{R}^{n}$, the \textit{convex feasibility problem} (CFP)
is to find a point in their intersection $C:=\cap_{j=1}^{m}C_{j}$,
assuming that the intersection is nonempty. This well-known problem
has applications in many theoretical and real-world scenarios, such
as image reconstruction from projections, data compression, radiation
therapy treatment planning, signal processing, sensor network source
localization, the solution of systems of linear or nonlinear inequalities 
induced by convex functions (since the solution of a system of inequalities
is a point in the intersection of the level-sets of the functions
which induce these inequalities), as well as in many other areas, as can be    
seen in, e.g., Cegielski's book \cite[p. 23]{Cegielski2012book}; see also  \cite{GibaliKuferReemSuss2018jour} and some of the references therein for the application of the CFP for solving optimization problems. Additional details about the CFP, including various algorithmic schemes for solving it and related references can be found in \cite{AharoniCensor1989jour,BauschkeBorwein1996jour,Bregman1967jour,Cegielski2012book,CensorReem2015jour,CensorZenios1997book,Combettes1996jour(CFP)}.

One of the methods for solving the CFP, which is of special importance
to our paper, is the BIP (Block-Iterative Projections) method of Aharoni
and Censor \cite{AharoniCensor1989jour}. In a nutshell, each iteration
of BIP is a relaxed convex combination of the orthogonal projections
onto the given subsets $C_{1},C_{2},\ldots,C_{m}$, where the combination's
weights themselves are dynamic, namely they may depend on the iteration
index and can vary from iteration to iteration. This method is rather
flexible since particular instances of it are fully sequential iterations
with repetitive controls, fully simultaneous iterations, and block
iterative iterations.

A more general problem is the \textit{common fixed point problem}
(CFPP) of finding a point in the intersection of the fixed point sets
of a finite family of operators $T_{i}:\mathbb{R}^{n}\to\mathbb{R}^{n}$,
$i\in\{1,2,\ldots,m\}$. This problem reduces to the CFP when for
each $i\in\{1,2,\ldots,m\}$, the operator $T_{i}$ is the orthogonal
projection onto $C_{i}$. Many methods have been devised to solve
the CFPP, under various settings: see, for example, \cite{Cegielski2012book,CensorSegal2009jour(CFPP),ReichZalas2016jour, Zaslavski2016book} and some of the references therein, as well as Algorithm \textbf{\ref{alg:GBIP}} and Examples \textbf{\ref{ex:FullySequential}}--\textbf{\ref{ex:BlockIterative}} below.

In this paper we consider the CFPP in the case where all operators
$T_{i}$ are continuous cutters. The class of continuous cutters is
quite wide and includes, among others, subgradient projections of 
differentiable convex functions having nonempty zero-level-sets,
resolvents of maximally monotone operators, and orthogonal projections
onto nonempty, closed and convex subsets of the space.

Cutters were introduced by Bauschke and Combettes in \cite{BauschkeCombettes2001jour} and by Combettes in  \cite{Combettes2001inproc} (not yet under the name “cutters” though). More details regarding this class of operators, as well as a short history and other names of it, can be found in Section \textbf{\ref{sec:NotationDefinitions}}
below. Our goal is to solve the CFPP asymptotically, namely to construct
an iterative sequence which converges to
a point in the common fixed point set of the given family of cutters.

We solve the CFPP using Algorithm \textbf{\ref{alg:GBIP}} which is introduced  below. This algorithm  is a generalization of the BIP method, mentioned above, where the generalization is expressed in the use of continuous cutters instead of orthogonal projections and in the permission of certain adaptive perturbations to appear in the iterative scheme (see Section \textbf{\ref{sec:VariantsBIP}} for other variants of the BIP method). As can be seen in both the formulation of Algorithm \textbf{\ref{alg:GBIP}} and in Examples \textbf{\ref{ex:FullySequential}}--\textbf{\ref{ex:BlockIterativeGeneralized}} below, Algorithm \textbf{\ref{alg:GBIP}} retains the flexibility of BIP not only because the users have freedom regarding the relaxation parameters and the weight functions, but also because Algorithm \textbf{\ref{alg:GBIP}} allows, as particular cases, fully sequential iterations with a repetitive control, fully simultaneous iterations, and block iterative iterations. Consequently, Algorithm \textbf{\ref{alg:GBIP}} can be adapted naturally to both serial and parallel computational architectures.

We show in Theorem \textbf{\ref{thm:GBIP}} below that the iterative sequence generated by Algorithm \textbf{\ref{alg:GBIP}} always converges globally to a generalized solution of the CFPP, namely to a point in a set which contains the common fixed point set of the given family of cutters. As we explain after this theorem (in Remark \textbf{\ref{rem:Q=00003DQhat}}), under mild conditions both sets coincide, and so, in a wide class of scenarios, the iterative sequence generated by our algorithmic scheme converges globally to a common fixed point of the given family of cutters. Along the way we obtain a result of independent interest, namely the apparently new Lemma \textbf{\ref{lem:PerturbedFejer}} below, related to Fej\'er-monotonicity in a perturbed form; this lemma shows that cutters, as well as relaxed versions of them, are not only quasi-nonexpansive, but rather their quasi-nonexpansivity is preserved under small perturbations
of a certain type.

As mentioned above, Algorithm \textbf{\ref{alg:GBIP}} allows perturbations
of a certain kind, but still converges (globally). In other words, our
algorithmic scheme exhibits a certain kind of resiliency, namely, it is ``perturbation resilient''. The perturbations in Algorithm
\textbf{\ref{alg:GBIP}} may appear as a result of noise, computational
errors, and so on. These perturbations may also be generated actively by the user as part of the application of the ``superiorization methodology'' (SM). In the SM, in contrast to the case in which the perturbations are unknown to the user (frequently only their magnitude can be estimated), the goal is to harness the permissible perturbations in order to obtain solutions, or generalized solutions, which are superior with respect to some given cost function, over (generalized) solutions which would be obtained without the generated perturbations. More details regarding the superiorization methodology, in its classical form, can be found in the initial papers  \cite{CensorDavidiHerman2010jour,DavidiHermanCensor2009jour}
and the survey papers \cite{Censor2015surv,Herman2014surv}; a re-examination
of this methodology, as well as a significant extension of its scope
and an extensive list of related references, can be found in \cite[Section 4]{ReemDe-Pierro2017jour}; a continuously updated bibliographical list of references related to the superiorization methodology can be found on-line in \cite{CensorSuperiorizationPage}.

A final word about potential computational advantages. This paper is a theoretical work. Comparative computational performance of BIP-for-cutters algorithms proposed and studied here can really be made only with exhaustive testing of the many possible specific variants permitted by the general schemes and their various user-chosen parameters. The computational advantages of the BIP algorithmic structures have been shown in the past for algorithms that use orthogonal projections rather than other cutter operators in many publications. For example, the work on proton computed tomography (pCT) in \cite{Karonis-2013} employs very efficiently a parallel code that uses a version of a block-iterative algorithm called “diagonally-relaxed orthogonal projections” (DROP), presented in \cite{DROP-2008}. See, e.g., also the recent paper on stochastic block projection algorithms by Necoara \cite{Necoara2022online}. It is plausible to hypothesize that since the BIP algorithmic structure and the cutter operators \cite{BauschkeCombettes2001jour}, see also \cite{Bauschke-cutters}, have been demonstrated to be computationally useful separately, then so might very well be their combination in the BIP-for-cutters scheme presented here. Admittedly, such practical questions should be resolved in future experimental works, preferably within the context of a significant real-world application.

\subsection{Paper layout\label{subsec:layout}}

Section \textbf{\ref{sec:NotationDefinitions}} presents the
notation and basic definitions used throughout the paper. Section
\textbf{\ref{sec:GBIP}} presents the generalized BIP method, namely
Algorithm \textbf{\ref{alg:GBIP}}, and further elaborates on it. A
few variants of the BIP method are discussed in Section \textbf{\ref{sec:VariantsBIP}},
where we compare them and their associated convergence results to
Algorithm \textbf{\ref{alg:GBIP}} and Theorem \textbf{\ref{thm:GBIP}}.
The convergence theorem (Theorem \textbf{\ref{thm:GBIP}}) and its
proof appear in Section \textbf{\ref{sec:Convergence}}.

\section{Notation and basic definitions\label{sec:NotationDefinitions}}

Given $n\in\mathbb{N}$, where $\mathbb{N}$ is the set of positive integers, let $X:=\mathbb{R}^{n}$ endowed with the Euclidean inner product $\langle\cdot,\cdot\rangle$ and the corresponding 
Euclidean norm $\|\cdot\|$. We denote by $d(x,C)$ the distance
between $x\in X$ and a nonempty set $C\subseteq X$, namely, $d(x,C):=\inf\{\| x-c\|\mid c\in C\}$,
and denote by $B[x,r]$ the closed ball with center $x$ and radius
$r\in[0,\infty]$ (of course, $B[x,r]=X$ if $r=\infty$ and $B[x,r]=\{x\}$
if $r=0$). The identity operator is denoted by $Id$, namely, $Id(x)=x$
for all $x\in X$. We use the convention that the sum over the empty
set is zero. Finally, for each operator $U:X\to X$, the set $\textnormal{Fix}(U):=\{x\in X | U(x)=x\}$ stands for the set of all fixed points of $U$. 

Given an operator $T:X\rightarrow X$ and a nonempty set $S\subseteq X$, we say that $T$ is a separator of $S$ provided 
\begin{equation}
\langle x-T(x),q-T(x)\rangle\leq0,\:\forall q\in S,x\in X.\label{eq:PropertyT}
\end{equation}
In particular, if $\textnormal{Fix}(T)$ is nonempty and $T$ is a separator of $S:=\textnormal{Fix}(T)$, then we say that $T$ is a \textit{cutter}.
Given $m\in\mathbb{N}$, denote $I:=\{1,2,\ldots,m\}$. A weight function
with respect to $I$ is a function $w:I\rightarrow[0,1]$ which satisfies
$\sum_{i\in I}w(i)=1$. Given a weight function $w:I\rightarrow[0,1]$
and a family $(T_{i})_{i\in I}$ of cutters, let $T_{w}$ be the operator
$T_{w}:X\rightarrow X$ defined by $T_{w}(x):=\sum_{i\in I}w(i)T_{i}(x)$,
for each $x\in X$.

The class of cutter operators was introduced in \cite{BauschkeCombettes2001jour} 
and \cite{Combettes2001inproc} under the name ``the class $\mathcal{T}$''.
Other names appear in the literature, for instance ``directed operators''
\cite{CensorSegal2009jour(split),Zaknoon2003PhD}. The name ``cutter''
was first suggested in \cite{CegielskiCensor2011inproc}. The reason
behind this name is that for each point $x$ in the space, which is
not a fixed point of $T$, the operator $T$ induces a hyperplane
(the one which is orthogonal to the vector $T(x)-x$ and passes through
$T(x)$) that ``cuts'' the space into two half-spaces: one of which
contains the fixed point set of $T$ and the other contains $x.$
Various examples of cutters can be found in \cite{BauschkeCombettes2001jour,Combettes2001inproc}
and (explicitly or implicitly) in \cite{BauschkeCombettes2017book} 
and \cite{Cegielski2012book}. A relatively recent work on cutters is \cite{Bauschke-cutters}.

In particular, the following operators are continuous cutters:
\begin{enumerate}
\item The subgradient projection of a (Fr\'echet) differentiable convex function $f:X\rightarrow\mathbb{R}$ whose zero-level-set $\{x\in X\mid f(x)\leq0\}$
is nonempty. Here 
\begin{equation}
T(x):=\left\{ \begin{array}{lll}
{\displaystyle {x-\frac{f(x)}{\|\nabla f(x)\|^2}\nabla f(x)},} & \text{if }\,f(x)>0,\\
x, & \text{if }\,f(x)\leq 0,
\end{array}\right.
\end{equation}
where $\nabla f(x)$ is, for each $x\in X$, the  gradient of $f$ at $x$. The subgradient inequality
and the assumption that $\{x\in X\mid f(x)\leq0\}\neq\emptyset$ imply
that $T$ is well-defined, namely, that $\nabla f(x)\neq 0$ if $f(x)>0$. Since $f$ is convex and differentiable, it is actually continuously differentiable \cite[Remark 6.2.6, p. 202]{Hiriart-UrrutyLemarechal2001book}. Hence from \cite[Proposition 29.41(ix), p. 553]{BauschkeCombettes2017book} it follows that $T$ is continuous, from \cite[Corollary 4.2.6, p. 146]{Cegielski2012book} 
it follows that $T$ is a cutter, and from \cite[Corollary 4.2.5, p. 145]{Cegielski2012book} it follows that the fixed point set of $T$ is $\{x\in\mathbb{R}^{n}\mid f(x)\leq0\}$. See also \cite{Bauschke-subgrad-2015};

\item Any firmly nonexpansive (FNE) operator, namely, $\|Tx-Ty\|^{2}\leq\langle Tx-Ty,x-y\rangle$
for all $x,y\in X$. Indeed, it is well-known that every firmly nonexpansive
is nonexpansive (see, for example, \cite[Theorem 2.2.10(v) or Theorem 2.2.10(vi), p. 70]{Cegielski2012book}).
Hence $T$ is continuous, and from \cite[Theorem 2.2.5]{Cegielski2012book} 
it follows that $T$ is a cutter;

\item An orthogonal projection on a nonempty, closed and convex subset $C$
of the space. Indeed, it is well-known that any orthogonal projection
is firmly nonexpansive (see, e.g., \cite[Proposition 4.16, p. 70]{BauschkeCombettes2017book}),
and so, as mentioned above, it is a continuous cutter. The fixed point
set of $T$ is $C$, as one can verify immediately. Several explicit
expressions for $T$, in some particular cases where $C$ has a simple
form, appear in \cite[Section 4.1]{Cegielski2012book};
\item The resolvent of a maximally monotone operator, namely, $T:=(Id+\gamma A)^{-1}$,
where $Id:X\rightarrow X$ is the identity operator, $\gamma>0$ and
$A:X\rightarrow2^{X}$ is a set-valued operator which is maximally
monotone. The assertion follows from \cite[Proposition 23.8(iii), p. 395]{BauschkeCombettes2017book},
\cite[Proposition 4.4(i),(v), p. 70]{BauschkeCombettes2017book} and
an elementary calculation. It is worth noting that the fixed point
set of $T$ is the zero set of $A$, namely, the set $\{x\in X\mid0\in Ax\}.$
This claim follows from an elementary calculation (see also \cite[Proposition 23.38, p. 405]{BauschkeCombettes2017book}).
\end{enumerate}

\section{The generalized BIP method}\label{sec:GBIP}

Under the assumptions and notations of Section \textbf{\ref{sec:NotationDefinitions}},
the generalized BIP algorithm is defined as follows: 

\begin{alg}\textbf{The generalized BIP method for cutters} \label{alg:GBIP}\\\\
\textbf{Input:} A positive integer $n$, an arbitrary initialization point
$x^{0}\in X:=\mathbb{R}^{n}$, two positive numbers $\tau_{1}$ and
$\tau_{2}$ which satisfy $\tau_{1}+\tau_{2}\leq2$, a positive integer $m$,
an index set $I:=\{1,2,\ldots,m\}$, a family of cutters $(T_{i})_{i\in I}$
defined on $X,$ with fixed point sets $Q_{i}:=Fix(T_{i})=\{x\in X\mid T_{i}(x)=x\}$
and a nonempty common fixed point set $Q:=\cap_{i\in I}Q_{i}$, a
(generalized) real number $\sigma\in(0,\infty]$ with the property
that $\sigma>d(x^{0},Q)$, a sequence of relaxation parameters $(\lambda_{k})_{k=0}^{\infty}$
which are positive numbers in the interval $[\tau_{1},2-\tau_{2}]$,
a sequence $(w_{k})_{k=0}^{\infty}$ of weight functions with respect
to $I$.\\\\
\textbf{Iterative step:} Given $k\in\mathbb{N}\cup\{0\}$ and the current
iterate $x^{k}$, calculate the next iterate $x^{k+1}$ by the iterative
process
\begin{equation}\label{eq:GBIP}
x^{k+1}:=x^{k}+\lambda_{k}(T_{w_{k}}(x^{k})-x^{k})+e^{k},
\end{equation}
where the error term $e^{k}\in X$ has the form 
\begin{equation}
e^{k}:=\sum_{i\in I}w_{k}(i)e^{k,i},
\end{equation}
and, for all $i\in I,$ the perturbation $e^{k,i}$ is any vector
in $X$ which satisfies 
\begin{equation}\label{eq:abs(eki)}
\|e^{k,i}\|\leq\frac{1}{2}\cdot\frac{\lambda_{k}(2-\lambda_{k})\|T_{i}(x^{k})-x^{k}\|^{2}}{\sqrt{\zeta_{k,i}}+\lambda_{k}\|T_{i}(x^{k})-x^{k}\|+2\sigma},
\end{equation}

where 
\begin{equation}
\zeta_{k,i}:=(\lambda_{k}\|T_{i}(x^{k})-x^{k}\|+2\sigma)^{2}+\lambda_{k}(2-\lambda_{k})\|T_{i}(x^{k})-x^{k}\|^{2}.
\end{equation}

\end{alg}

\begin{example} \label{ex:FullySequential} Algorithm \textbf{\ref{alg:GBIP}}
becomes fully sequential if for every $k\in\mathbb{N}$, one has $w_{k}(i)=0$
for all $i\in I$ with the exception of one index $i_{0}(k)$ for
which $w_{k}(i_{0}(k))=1$. In this case the index $i_{0}$ can be
regarded as a control function that maps $\mathbb{N}\cup\{0\}$ to
$I$ by assigning to the given index $k\in\mathbb{N}\cup\{0\}$ the
unique index $i_{0}(k)\in I$. If $i_{0}$ has the property that for
all $j\in I$ there are infinitely many $k\in\mathbb{N}$ such that
$i_{0}(k)=j$, namely $w_{k}(j)=1$, then $i_{0}$ is the so-called
\textit{repetitive control}. Well-known particular cases of repetitive
controls are cyclic and almost cyclic controls which are, in turn,
also special cases of the class of \textit{expanding controls} presented
in \cite{expanding-2011}. 

In particular, Algorithm \textbf{\ref{alg:GBIP}} generalizes the
well-known method of successive orthogonal projections. Moreover,
in the case of repetitive controls any $j\in I$ satisfies $\sum_{k=1}^{\infty}w_{k}(j)=\infty$,
and, hence, Theorem \textbf{\ref{thm:GBIP}} below ensures that the
iterative sequence $(x^{k})_{k\in\mathbb{N}}$ converges to a point
in the intersection of the fixed point sets of the given family of
operators $(T_{i})_{i\in I}$. \end{example}

\begin{example} \label{ex:FullySimultaneous} Algorithm \textbf{\ref{alg:GBIP}}
becomes fully simultaneous when $w_{k}(i)>0$ for all $i\in I$ and
$k\in\mathbb{N}\cup\{0\}$, since in this case at each iteration all
the cutters $(T_{i})_{i\in I}$ are considered. If, in addition, all
the cutters are orthogonal projections onto given hyperplanes, then
Algorithm \textbf{\ref{alg:GBIP}} becomes a Cimmino-type algorithm
for solving the linear system induced by these hyperplanes. \end{example}

\begin{example} \label{ex:BlockIterative} Algorithm \textbf{\ref{alg:GBIP}} becomes block-iterative in the classical sense if the following scenario
occurs: first, one partitions the given index set $I:=\{1,2,\ldots,m\}$
into $\widetilde{m}\leq m$ ``blocks'', namely, into $\widetilde{m}\in\mathbb{N}$
disjoint and nonempty index subsets $I_{1},I_{2},\ldots,I_{\widetilde{m}}$
whose union is $I$; then one defines a control function over the
block indices, namely a function $\widetilde{c}:\mathbb{N}\cup\{0\}\to\{1,2,\ldots,\widetilde{m}\}$;
then, for each $k\in\mathbb{N}\cup\{0\}$, one defines a weight function
$w_{k}:I\to[0,1]$ by $w_{k}(i):=0$ if $i\notin I_{\widetilde{c}(k)}$,
and $w_{k}(i)$ an arbitrary number in $[0,1]$ if $i\in I_{\widetilde{c}(k)}$,
with the additional condition that $\sum_{i\in I_{\widetilde{c}(k)}}w_{k}(i)=1$.

For instance, suppose that $\widetilde{c}(k)=(k\mod\widetilde{m})+1$
for all $k\in \mathbb{N}\cup\{0\}$; suppose further that for each $\widetilde{j}\in\{1,2,\ldots,\widetilde{m}\}$
there are $\alpha_{\widetilde{j}}$ elements in block number $\widetilde{j}$;
if one defines $w_{k}(i):=0$ when $i\notin I_{\widetilde{c}(k)}$
and $w_{k}(i):=1/\alpha_{\widetilde{c}(k)}$ when $i\in I_{\widetilde{c}(k)}$,
then this is a control which cycles periodically between the blocks;
now, in order to construct $x^{k+1}$, one first observes that $\widetilde{c}(k)=\widetilde{j}$
for some $\widetilde{j}\in\{1,2,\ldots,\widetilde{m}\}$, then one
considers all the cutters in block number $\widetilde{j}$ and gives
them an equal weight $1/\alpha_{\widetilde{j}}$, then one constructs
the weighted sum $T_{w_{k}}$ of the cutters in that block, and from
$T_{w_{k}}$ and \textbf{(\ref{eq:GBIP})} one obtains $x^{k+1}$.

\end{example}

\begin{example}\label{ex:BlockIterativeGeneralized}

Algorithm \ref{alg:GBIP} becomes block-iterative in the generalized
sense if the following scenario occurs: one defines a block selection
function $J:\mathbb{N}\cup\{0\}\to2^{I}\backslash\left\{ \emptyset\right\} $
which, at iteration number $k$, selects a block $J_{k},$ namely
a nonempty subset of $I$; then, for each $k\in\mathbb{N}\cup\{0\}$,
one defines a weight function $w_{k}:I\to[0,1]$ by $w_{k}(i):=0$
if $i\notin J_{k}$, and $w_{k}(i)$ is an arbitrary number in $[0,1]$
if $i\in J_{k}$, with the additional condition that $\sum_{i\in J_{k}}w_{k}(i)=1$.
Under these assumptions, \textbf{(\ref{eq:GBIP})} becomes

\begin{equation}
x^{k+1}:=x^{k}+\lambda_{k}\sum_{i\in J_{k}}w_{k}(T_{i}(x^{k})-x^{k})+\sum_{i\in J_{k}}w_{k}e^{k,i}.
\end{equation}

\end{example}

\begin{remark} The condition described in \textbf{(\ref{eq:abs(eki)})}
is an adaptive one. It seems to be new, although it is inspired from
other forms of adaptive error terms which appear in \cite[Section 5]{CensorReem2015jour},
\cite[Subsection 2.3]{ReemDe-Pierro2017jour}. It is unclear whether
the sequence $(e^{k})_{k=0}^{\infty}$ is summable, and hence convergence results
which discuss \textbf{(\ref{eq:GBIP})} with summable errors cannot be used.
\end{remark}

\begin{remark} \label{rem:sigma} The (generalized) real number $\sigma$
given in the input of Algorithm \textbf{\ref{alg:GBIP}} poses a certain
limitation on the error terms $e^{k,i}$ and $e^{k}$, for all $i\in I$
and $k\in\mathbb{N}\cup\{0\}$. Indeed, one has to be able to derive
an estimate on how far is the solution set $Q$ located from the initial
iteration vector $x^{0}$ in order to have in hand an explicit $\sigma$.
Such an explicit estimate can be derived sometimes.

For example, if one is able to show that $Q$ is bounded, i.e., that
it is strictly contained inside some ball $B[c^{0},r]$, then the
triangle inequality implies that any $\sigma\in(r+\|x^{0}-c^{0}\|,\infty)$
satisfies $d(x^{0},Q)<\sigma$. Such a case obviously occurs when,
for instance, $Q_{i}$ is bounded for some $i\in I$, or $Q_{i}\cap Q_{j}$
is bounded for some $i,j\in I$. Real-world scenarios in which $Q$
is bounded occur, for example, in sensor network source localization
problems in acoustics \cite{HeroBlatt2005inproc} and in wireless (electromagnetic) communication \cite{GholamiWymeerschStromRydstrom2011jour}, since
in both cases actually all the sets $Q_{i}$ are bounded (they are
discs).

As another example, consider the case of a consistent linear equation
$Ax=y$, where $s\in\mathbb{N}$, $A\in\mathbb{R}^{s\times n}$ and
$y\in\mathbb{R}^{s}$ are given and the desired solution $x\in\mathbb{R}^{n}$
should satisfy the additional constraint $\|x\|_{1}\leq\varepsilon$
for some given $\varepsilon>0$, where $\|x\|_{1}:=\sum_{i=1}^{n}\arrowvert  x_i\arrowvert$
is  the $\ell_{1}$-norm of $x=(x_{i})_{i=1}^{n}$. Such a problem
has applications in signal processing \cite{CarmiCensorGurfil2012jour}.
Since $\|x\|\leq\|x\|_{1}$ always holds (where $\|\cdot\|$ is the
Euclidean norm), one has $\|x\|\leq\varepsilon$; hence, from the
triangle inequality, $d(x^{0},Q)\leq\|x^{0}-x\|\leq\|x^{0}\|+\|x\|\leq\|x^{0}\|+\varepsilon$;
thus, any $\sigma>\|x^{0}\|+\varepsilon$ is good for the purpose
of Algorithm \textbf{\ref{alg:GBIP}}.

Anyway, if one is unable to estimate $d(x^{0},Q)$ from above, then
one may be forced to assume that $\sigma=\infty$, which means that
all the error terms vanish. 
\end{remark}

\begin{remark}
Algorithm \textbf{\ref{alg:GBIP}}, as described above, continues forever. Of course, in practice one needs some terminating condition in order to obtain an output. One such a criterion can be to stop the iterative process at some large iteration, say $k=10^6$, and to take the corresponding point $x^k$ as the output. Another criterion can be to check, in each iteration, the distance $d(x^k,Q_i)$ for each $i\in I$, assuming these distances can be evaluated, and to stop the process when $\max\{d(x^k,Q_i) | i\in I\}\leq \widehat{\epsilon}$ for some predetermined $\widehat{\epsilon}\geq 0$ (the case $\widehat{\epsilon}:=0$ is of interest only when one can prove convergence to $\cap_{i\in I}Q_i$ after finitely many iterations). A third criterion, at least in the case where all the sets $Q_i$ are zero-level-sets of some functions $f_i$, is to evaluate, in each iteration, $f_i(x^k)$ for all $i\in I$, and to stop the process when $\max\{f_i(x^k): i\in I\}\leq \bar{\epsilon}$ for some predetermined $\bar{\epsilon}\geq 0$. Other terminating conditions can be given. 
\end{remark}

\section{Variants of BIP\label{sec:VariantsBIP}}

Over the years other variations of the BIP method have
appeared. We discuss the ones which we are aware of in this section, where we also make a few comparisons between them and our method and convergence result.
We focus on variants in which the considered operators are cutters in general and not just particular cases of them such as orthogonal projections or firmly nonexpansive operators. For the sake of completeness, we also mention briefly,
in the last item of the list, variants of BIP of this latter type,
as well as corresponding convergence results.

In what follows $\widehat{I}:=\{i\in I\mid\sum_{k=0}^{\infty} w_{k}(i)=\infty\}$
and $\widehat{Q}:=\cap_{i\in\widehat{I}}\,Q_{i}$, with the convention
that $\widehat{Q}:=X$ if $\widehat{I}=\emptyset$. Here is our list.

\textbf{(1)} The original BIP method appears in \cite[Algorithm 1]{AharoniCensor1989jour}. There the space is finite-dimensional, the finitely many cutters are orthogonal projections onto given nonempty, closed and convex subsets and no perturbations are allowed. Our proof is inspired by \cite[Theorem 1]{AharoniCensor1989jour}, but because of the different settings, there are several significant differences between our proof and the proof which appears in \cite{AharoniCensor1989jour}; for instance, we  need Lemma \textbf{\ref{lem:PerturbedFejer}} and also need to perform a careful analysis in Lemma \textbf{\ref{lem:beta}} as a result of the appearance of perturbations. 

\textbf{(2)} In \cite[Chapter 2]{Zaknoon2003PhD}, and in the unpublished
technical report \cite{CensorZaknoon2003Draft} (albeit some modifications
are needed there), appears a version of \cite[Algorithm 1]{AharoniCensor1989jour} 
in which the index set $I$ is finite, the space $X$ is a finite-dimensional
Euclidean space, the operators $(T_{i})_{i\in I}$ are continuous
cutters, perturbations are not allowed (namely, they vanish), and
the relaxation parameters $\lambda_{k}$ satisfy the condition $\lambda_{k}\in[\tau_{1},(2-\tau_{2})L(x^{k},w_{k})]$
for all $k\in\mathbb{N}\cup\{0\}$, where, for all $k\in\mathbb{N}\cup\{0\},$
\begin{equation}
L(x^{k},w_{k}):=\left\{ \begin{array}{lll}
{\displaystyle 1,} & \text{if }\,x^{k}=T_{w_{k}}(x^{k}),\\
\displaystyle{\sum_{i\in I}\frac{w_{k}(i)\|T_{i}(x^{k})-x^{k}\|^{2}}{\|T_{w_{k}}(x^{k})-x^{k}\|^{2}}}, & \text{otherwise. }
\end{array}\right.
\end{equation}
 From the convexity of the square norm, it follows that $L(x^{k},w_{k})\geq1$
for all $k\in\mathbb{N}\cup\{0\}$. It is also assumed that $Q\neq\emptyset$
and $\widehat{I}=I$ (hence $Q=\widehat{Q}$).

The first convergence theorem is \cite[Theorem 2.4.11]{Zaknoon2003PhD}  
(essentially \cite[Theorem 20]{CensorZaknoon2003Draft}), which says
that if the interior of $Q$ is nonempty, then the algorithmic sequence
converges to a point in $Q$. The second convergence theorem is \cite[Theorem 2.4.12]{Zaknoon2003PhD} 
(essentially \cite[Theorem 21]{CensorZaknoon2003Draft}), which says
that if merely $\lambda_{k}\in[\tau_{1},2-\tau_{2}]$ for all $k\in\mathbb{N}\cup\{0\}$,
then the algorithmic sequence converges to a point in $Q$. The third
convergence theorem is \cite[Theorem 2.5.3]{Zaknoon2003PhD} (essentially 
\cite[Theorem 24]{CensorZaknoon2003Draft}), which says that if $Q_{i}$
is strictly convex for all $i\in I$, then the algorithmic sequence
converges to a point in $Q$. The fourth convergence theorem is \cite[Theorem 2.5.4]{Zaknoon2003PhD}  
(essentially \cite[Theorem 25]{CensorZaknoon2003Draft}), in which
it is assumed that the sequence of weight functions $(w_{k})_{k=0}^{\infty}$
is fair (see Remark \textbf{\ref{rem:fair}} below), and there is
some fixed positive number $\xi$ such that for all $i\in I$ and
all $k\in\mathbb{N}\cup\{0\}$, if $w_{k}(i)>0$, then actually $w_{k}(i)>\xi$;
under these assumptions the theorem says that the algorithmic sequence
converges to a point in $Q$.

The technique used in \cite[Chapter 2]{Zaknoon2003PhD} and  \cite{CensorZaknoon2003Draft}
for establishing the convergence results has several similarities
to the technique used here, but there are also some differences, partly
because the settings are not identical. Examples of differences are
the use of the Pierra's product-space formulation in \cite{Zaknoon2003PhD,CensorZaknoon2003Draft}  
and not here, the use of Lemma \textbf{\ref{lem:PerturbedFejer}}
here (which is a new lemma not used elsewhere), the need to handle
extrapolations in \cite{Zaknoon2003PhD,CensorZaknoon2003Draft} and
perturbations here, etc.

\textbf{(3)} Algorithm 6.1 in \cite{Combettes2001inproc} is a general
variant of \cite[Algorithm 1]{AharoniCensor1989jour}, in which the
setting is a real Hilbert space, certain kind of perturbations are
allowed (essentially summable), cutters are used instead of just orthogonal
projections, and one allows an infinite index set $I$ where in each
iteration $k$ the sum is over a nonempty and finite subset $I_{k}$
of $I$ (namely, this algorithmic scheme is block-iterative in the
sense of Example \textbf{\ref{ex:BlockIterativeGeneralized}}, but with the
modification that the range of the selection function is not $2^{I}\backslash\left\{ \emptyset\right\} $
but rather the set of all nonempty and finite subsets of $I$). In
\cite[Theorem 6.6]{Combettes2001inproc} it is proved that the sequence
converges weakly to the feasible set, and under stronger assumptions
strong convergence holds.

On the other hand, a stronger assumption is assumed there, namely,
\cite[Algorithm 6.1, Part 4]{Combettes2001inproc} which says that
there is a fixed positive number $\delta_{1}$ such that in each iteration
one of the weights, which corresponds to an index $j$, is at least
as large as $\delta_{1}$, and at this same index $j$ another technical
condition holds (that is, $\|T_{j,k}x^{k}-x^{k}\|=\max_{i\in I_{k}}\|T_{i,k}x^{k}-x^{k}\|$,
where $T_{i,k}$ is the $i$-th operator at iteration $k$ and where
$i$ is taken from the index set $I_{k}$).

Moreover, in the relevant convergence result \cite[Theorem 6.6]{Combettes2001inproc} 
one assumes that the control sequence $(I_{k})_{k=0}^{\infty}$ is
admissible (which is a general condition, but weaker than a repetitive
control). Neither in \cite[Algorithm 1 and Theorem 1]{AharoniCensor1989jour}  
nor in Algorithm \textbf{\ref{alg:GBIP}} and Theorem \textbf{\ref{thm:GBIP}}
here these assumptions are imposed. Furthermore, the convergence in
\cite[Theorem 6.6]{Combettes2001inproc} is to the feasible set $Q$
rather than to $\widehat{Q}$ as in \cite[Theorem 1]{AharoniCensor1989jour}  
and in Theorem \textbf{\ref{thm:GBIP}} below (the equality $\widehat{Q}=Q$
holds under mild conditions, which in particular hold under the assumptions
in \cite[Theorem 6.6]{Combettes2001inproc}, but in general $Q\subset\widehat{Q}$).

\textbf{(4)} The setting in \cite[Theorem 5.8.15]{Cegielski2012book} 
and \cite[Theorem 9.27]{CegielskiCensor2011inproc} (both results
are essentially identical) is a possibly infinite-dimensional real
Hilbert space and not necessarily continuous cutters and the sum in
each iteration is over a nonempty subset $J_{k}$ of the finite index
set $I$; the cutters $T_{i}^{k}$ in the sum are dynamic, namely
they depend on both the iteration index $k$ and the sum index $i$;
however, these cutters should satisfy certain conditions, such as
the existence of a fixed and finite family $U_{i},i\in I$ of cutters
with a nonempty common fixed point set $Q$ such that $\bigcap_{i\in J_{k}}Q_{i}^{k}\supseteq Q$
(where $Q_{i}^{k}$ is the fixed point set of $T_{i}^{k}$) and also
that $U_{i}-Id$ is demi-closed at 0 for all $i\in I$. 

Under further assumptions, such as approximate regularity of the weight
functions, it is shown that the algorithmic sequence converges weakly
to $Q$, and if the space is finite-dimensional and one assumes less
(semi-regularity of the weight functions), then the algorithmic sequence
converges to $Q.$ The convergence result of Aharoni and Censor \cite[Algorithm 1]{AharoniCensor1989jour},
is essentially obtained as a consequence of \cite[Theorem 5.8.15]{Cegielski2012book} 
or \cite[Theorem 9.27]{CegielskiCensor2011inproc}, and is illustrated,
respectively, in \cite[Example 5.8.18]{Cegielski2012book} and \cite[Example 9.30]{CegielskiCensor2011inproc},
for the special case of orthogonal projections onto nonempty, closed
and convex subsets of the space.

It is worthwhile to note that no perturbations are allowed in \cite[Theorem 5.8.15]{Cegielski2012book}, 
\cite[Example 5.8.18]{Cegielski2012book}, \cite[Theorem 9.27]{CegielskiCensor2011inproc} 
and \cite[Example 9.30]{CegielskiCensor2011inproc}, and while it
seems that the method of \cite[Example 5.8.18]{Cegielski2012book}  
and \cite[Theorem 9.27]{CegielskiCensor2011inproc} can be generalized
to other cutters (by modifying the arguments in \cite[Example 26(e)]{CegielskiCensor2011inproc}),
it does not seem that it can be generalized to the perturbations that we consider
in Algorithm \textbf{\ref{alg:GBIP}}, since the method of \cite[Example 5.8.18]{Cegielski2012book} 
is heavily based on a certain nonnegative (usually positive) lower
bound on $\|x^{k+1}-q\|-\|x^{k}-q\|$, and this lower bound is eliminated
when the perturbations that we consider in Algorithm \textbf{\ref{alg:GBIP}}
appear. 

We also note that one can find in both \cite{Cegielski2012book} and 
\cite{CegielskiCensor2011inproc} other results which are closely
related to \cite[Theorem 5.8.15]{Cegielski2012book} and \cite[Theorem 9.27]{CegielskiCensor2011inproc},
such as \cite[Theorem 5.10.2]{Cegielski2012book} and \cite[Theorem 9.35]{CegielskiCensor2011inproc} 
(admissible step sizes), and \cite[Theorem 5.8.25]{Cegielski2012book}  
(for orthogonal projections), where in all of these cases no perturbations
appear.

\textbf{(5)} The setting in \cite[Theorem 4.1 and Theorem 4.5]{ReichZalas2016jour}  
is a possibly infinite-dimensional real Hilbert space, not necessarily
continuous cutters, but ones which should satisfy other conditions,
such as the Opial's demi-closedness principle (for the weak convergence
case); the algorithmic scheme allows strings and not just convex combinations
and relaxations as in Algorithm \textbf{\ref{alg:GBIP}} above. On
the other hand, in both \cite[Theorem 4.1]{ReichZalas2016jour} and 
\cite[Theorem 4.5]{ReichZalas2016jour} there is a certain restriction
on the control, namely, Condition (ii) there which is something in
the spirit of an almost cyclic control, and the whole convergence
is to the common fixed point set, while in our Theorem \textbf{\ref{thm:GBIP}}
such a restriction does not exist and the convergence is not necessarily
to a point in the common fixed point set, but rather to a point located
in the possibly larger set $\widehat{Q}$.

As an illustration to this last point, consider the case of strings
of length one and weights which, in each iteration, vanish with the
exception of one place in which they are equal to 1. This is the case
of a fully sequential algorithmic scheme, and Condition (ii) in \cite[Theorem 4.1]{ReichZalas2016jour}  
is the classical almost cyclic control. On the other hand, in our
Theorem \textbf{\ref{thm:GBIP}} one allows the control to be repetitive,
that is, more general, as explained in Example \textbf{\ref{ex:FullySequential}}
above. No perturbations are allowed in \cite[Theorem 4.1]{ReichZalas2016jour},
while in \cite[Theorem 4.5]{ReichZalas2016jour} summable perturbations
are allowed but the operators must be firmly nonepxansive, and hence
(see Section \textbf{\ref{sec:NotationDefinitions}}) must be continuous
cutters.

\textbf{(6)} The setting in \cite[Theorem 3.1, Theorem 3.2]{KolobovReichZalas2017jour} 
is a possibly infinite-dimensional real Hilbert space and cutters
which are not necessarily continuous, and the sum in each iteration
is over a nonempty subset $I_{k}$ of the finite index set $I$ (namely,
this algorithmic scheme is block-iterative in the generalized sense
of Example \textbf{\ref{ex:BlockIterativeGeneralized}}). The cutters, however,
should satisfy other conditions which we do not impose in Algorithm
\textbf{\ref{alg:GBIP}}, such as having a representation to their
fixed point sets as the zero-level-sets of well-behaved proximity
functions. 

Additional conditions which are not imposed in Algorithm \textbf{\ref{alg:GBIP}}
but are imposed in \cite{KolobovReichZalas2017jour} are that the
upper bound $\tau_{2}$ on the relaxation parameters should be at
least 1, and that there is a positive number $\omega^{-}\in[0,1]$
such that $w_{k}(i)\geqq\omega^{-}$ for all $k\in\mathbb{N}\cup\{0\}$
and all $i\in I$, namely the algorithmic scheme in \cite{KolobovReichZalas2017jour}  
is fully simultaneous with a strictly positive lower bound on the
weights. In addition, one needs to impose there an assumption (Condition
(ii)) which is essentially an almost cyclic control and no perturbations
are allowed there. Under these and additional assumptions, the authors
of \cite{KolobovReichZalas2017jour} derive weak, strong and linear
convergence of the iterative algorithmic scheme to a common fixed
point of the given cutters.

\textbf{(7)} Other variants of BIP, for more restricted cutters or for other types of operators, as well as associated convergence results, appear in the following publications.
\cite[The algorithmic scheme on p. 378, Theorem 3.20, Corollary 3.22, Corollary 3.24, Corollary 3.25]{BauschkeBorwein1996jour}:
Here finite and infinite-dimensional real Hilbert spaces with firmly nonexpansive operators are considered (many other convergence results for orthogonal projections in \cite[Sections 4--6]{BauschkeBorwein1996jour}); \cite[Algorithm (2.2), Theorems 1, 2]{FlamZowe1990jour}:
The finite-dimensional case with projections onto separating hyperplanes;
\cite[Theorem 4.4]{ButnariuCensor1990jour}: This is an almost simultaneous
BIP method for orthogonal projections in a finite-dimensional Euclidean
space; \cite[The method of (1), Theorem 1]{ButnariuCensor1994jour}:
This is an almost simultaneous BIP method for orthogonal projections
in an infinite-dimensional Hilbert space; \cite[The method of (8), Theorem 4.1]{AleynerReich2008jour}:
Firmly nonexpansive mappings in finite-dimensional strictly convex
normed spaces; \cite[Algorithm 6.5, Theorem 6.4]{Combettes2000jour}:
A modified BIP method with orthogonal projections in Hilbert spaces;
\cite[Theorem 4.1 and some theorems in Section 5]{IbarakiTakahashi2008jour}:
a modified fully simultaneous BIP method with generalized nonexpansive
mappings in smooth and uniformly convex Banach spaces.

\section{The global convergence theorem and its proof}\label{sec:Convergence}
In this section we formulate and prove our global convergence theorem concerning Algorithm \textbf{\ref{alg:GBIP}}.

\begin{theorem} \label{thm:GBIP} Under the notations and assumptions
of Sections \textbf{\ref{sec:NotationDefinitions}} and \textbf{\ref{sec:GBIP}},
assume that $T_{i}$ is continuous for all $i\in I$. Denote $\widehat{I}:=\{i\in I\mid\sum_{k=0}^{\infty} w_{k}(i)=\infty\}$
and $\widehat{Q}:=\cap_{i\in\widehat{I}}\,Q_{i}$, with the convention
that $\widehat{Q}:=X$ if $\widehat{I}=\emptyset$. Then any sequence
defined in \textbf{(\ref{eq:GBIP})} converges to a point in $\widehat{Q}\cap B[x^{0},2\sigma]$.
In particular, if $\widehat{I}=I$, then the sequence defined in \textbf{(\ref{eq:GBIP})}
converges to a point in $Q\cap B[x^{0},2\sigma]$. \end{theorem}

The proof of Theorem \textbf{\ref{thm:GBIP}} is based on several
claims which are formulated and proved below. Before proceeding with
these claims, we need a further notation: given $x\in X$, $q\in X$,
$\lambda\in[0,2]$, $\theta\in[0,\infty)$, and $i\in I$, we denote
by $E_{\theta}(x,q,\lambda,i)$ the set of all $e\in X$ which satisfy

\begin{equation}\label{eq:abs(e)}
\|e\|\leq\frac{\theta\cdot\lambda(2-\lambda)\|T_{i}(x)-x\|^{2}}{\sqrt{\zeta}+\lambda\|T_{i}(x)-x\|+\|x-q\|}
\end{equation}

with 
\begin{equation}
\zeta:=(\lambda\|T_{i}(x)-x\|+\|x-q\|)^{2}+\lambda(2-\lambda)\|T_{i}(x)-x\|^{2},
\end{equation}

where both sides of \textbf{(\ref{eq:abs(e)})} mean zero if the denominator of the fraction on the right-hand side, and hence also the numerator, vanish. In addition, given a weight function $w:I\to[0,1]$, we denote by $E_{\theta}(x,q,\lambda,w)$ the set of all $e\in X$ which satisfy $e=\sum_{i\in I}w(i)e^{i}$, where $e^{i}\in E_{\theta}(x,q,\lambda,i)$ for all $i\in I$.

We start with the following apparently new lemma, which seems to be
of independent interest. Since it enables us to prove the (essentially)
Fej\'er monotonicity of the sequence (Lemma \textbf{\ref{lem:Fejer-kQ}}
below), and since a certain perturbation appears in it (the term $e$),
Lemma \textbf{\ref{lem:PerturbedFejer}} can be thought of as establishing
a Fej\'er monotonicity phenomenon in a perturbed form. As explained
in Remark \textbf{\ref{rem:StableQuasiNE}} below, Lemma \textbf{\ref{lem:PerturbedFejer}}
actually shows that the operator $Id+\lambda(T-Id)$ is not only quasi-nonexpansive,
but rather that its quasi-nonexpansiveness is stable under small perturbations,
and that this phenomenon holds in a general setting.

\begin{lemma} \label{lem:PerturbedFejer}
{\bf (Fej\'er monotonicity in a perturbed form):} Suppose that $T:X\to X$ is a separator of a
nonempty subset $S\subseteq X$. Given $x\in X$, $q\in S$ and $\lambda\in[0,2]$,
if $e\in X$ satisfies (\textbf{\ref{eq:abs(e)}}) with $\theta:=1$
and $T$ instead of $T_{i}$, and if 
\begin{equation}\label{eq:y}
y:=x+\lambda(T(x)-x)+e,
\end{equation}
then 
\begin{equation}\label{eq:PerturbedFejer}
\|y-q\|\leq\|x-q\|.
\end{equation}
If, in addition, $x\neq T(x)$, $0<\lambda<2$ and $e\in X$ satisfies
\textbf{(\ref{eq:abs(e)})} with strict inequality, then 
\begin{equation}\label{eq:PerturbedFejerStrict}
\|y-q\|<\|x-q\|.
\end{equation}
\end{lemma}

\begin{proof} As a result of \textbf{(\ref{eq:y})}, \textbf{(\ref{eq:PropertyT})},
the Cauchy-Schwarz inequality and the assumption that $\lambda\in[0,2]$,
we have

\begin{multline}
\|y-q\|^{2}=\|x-q\|^{2}+\|\lambda(T(x)-x)+e\|^{2}+2\langle x-q,\lambda(T(x)-x)+e\rangle\\
=\|x-q\|^{2}+\lambda^{2}\|T(x)-x\|^{2}+\|e\|^{2}+2\lambda\langle T(x)-x,e\rangle+2\lambda\langle x-q,T(x)-x\rangle+2\langle x-q,e\rangle\\
=\|x-q\|^{2}+\lambda^{2}\|T(x)-x\|^{2}+2\lambda\langle T(x)-q,T(x)-x\rangle-2\lambda\langle T(x)-x,T(x)-x\rangle\\
+2\lambda\langle T(x)-x,e\rangle+2\langle x-q,e\rangle+\|e\|^{2}\\
=\|e\|^{2}+2\lambda\langle T(x)-x,e\rangle+2\langle x-q,e\rangle+\|x-q\|^{2}-\lambda(2-\lambda)\|T(x)-x\|^{2}\\
+2\lambda\langle x-T(x),q-T(x)\rangle\\
\leq\|e\|^{2}+2(\lambda\|T(x)-x\|+\|x-q\|)\|e\|-\lambda(2-\lambda)\|T(x)-x\|^{2}+\|x-q\|^{2}\\
\leq\|x-q\|^{2}.\label{eq:abs(y-q)<=abs(x-q)}
\end{multline}
To derive the last inequality in \textbf{(\ref{eq:abs(y-q)<=abs(x-q)})}
we used the following simple facts: (i) from elementary analysis and
algebra, given two nonnegative numbers $\alpha_{1}$ and $\alpha_{2}$,
the inequality $t^{2}+2\alpha_{1}t-\alpha_{2}\leq0$ for nonnegative
$t$ holds whenever $t\in[0,\sqrt{\alpha_{1}^{2}+\alpha_{2}}-\alpha_{1}]$;
(ii) the simple identity $\sqrt{\alpha_{1}^{2}+\alpha_{2}}-\alpha_{1}=\alpha_{2}/(\sqrt{\alpha_{1}^{2}+\alpha_{2}}+\alpha_{1})$
holds; (iii) the equation \textbf{(\ref{eq:abs(e)})} actually says that
$t\leq\alpha_{2}/(\sqrt{\alpha_{1}^{2}+\alpha_{2}}+\alpha_{1})$ for
$t:=\|e\|$, $\alpha_{1}:=\lambda\|T(x)-x\|+\|x-q\|$ and $\alpha_{2}:=\lambda(2-\lambda)\|T(x)-x\|^{2}$;
(iv) the last inequality in \textbf{(\ref{eq:abs(y-q)<=abs(x-q)})}
can be written as $t^{2}+2\alpha_{1}t-\alpha_{2}+\|x-q\|^{2}\leq\|x-q\|^{2}$.

Finally, if $x\neq T(x)$ and $0<\lambda<2$, then the number on the
right-hand side of \textbf{(\ref{eq:abs(e)})} is positive. Hence 
there are vectors $e\in X$ fulfilling the strict version of  \textbf{(\ref{eq:abs(e)})}, namely, these vectors are all the ones  whose magnitudes are smaller than the right-hand side of \textbf{(\ref{eq:abs(e)})}. Let $e$ be such a vector. This means that in the notation of the previous paragraph, 
$0\leq t<\alpha_{2}/(\sqrt{\alpha_{1}^{2}+\alpha_{2}}+\alpha_{1})=\sqrt{\alpha_{1}^{2}+\alpha_{2}}-\alpha_{1}$
and $\alpha_{2}>0$. These inequalities and elementary properties
of quadratic inequalities imply that $t^{2}+2\alpha_{1}t-\alpha_{2}<0$.
Thus, the last inequality in \textbf{(\ref{eq:abs(y-q)<=abs(x-q)})}
is strict. \end{proof}

\begin{remark} \label{rem:StableQuasiNE}Lemma \textbf{\ref{lem:PerturbedFejer}}
is rather general, since the space $X$ can be an arbitrary real inner
product space and not necessarily a finite-dimensional Euclidean space,
and the operator $T$ there is not necessarily a cutter and not necessarily
continuous. Moreover, this lemma shows that the operator $T$ which
appears there, and also a relaxed version of it, exhibit a certain
stability property.

Indeed, we recall that an operator $T:X\to X$ with a nonempty fixed
point set $Fix(T)$ is called quasi-nonexpansive if $\|Tx-q\|\leq\|x-q\|$
for all $x\in X$ and $q\in Fix(T)$. Now, if we take $e:=0$ and
$\lambda:=1$ in Lemma \textbf{\ref{lem:PerturbedFejer}} then it
follows that any cutter $T$ is quasi-nonexpansive. Furthermore, Lemma
\textbf{\ref{lem:PerturbedFejer}} shows that the operator $T_{\lambda}:=Id+\lambda(T-Id)$
is quasi-nonexpansive for every $\lambda\in[0,2]$, where $Id$ is
the identity operator, and, as a matter of fact, the property of being
quasi-nonexpansive holds true even if we translate $T_{\lambda}$
by a vector $e$ which satisfies \textbf{(\ref{eq:abs(e)})} with $\theta:=1$.
In other words, the property of $T_{\lambda}$ being quasi-nonexpansive
is stable under certain small perturbations. 
\end{remark}

\begin{lemma}\label{lem:eta} 
Let $z\in X$ be given, and denote
$I_{z}:=\{i\in I\mid z\notin Q_{i}\}.$ Suppose that $G$ is a nonempty
and compact subset of $X$. Then there exists an $\eta\in[0,\infty)$
such that for all $x\in G$, all $\lambda\in[0,2]$, all weight functions
$w:I\to[0,1]$ and all $e\in E_{1}(x,z,\lambda,w)$, one has 
\begin{equation}\label{eq:Sum-eta}
\|x+\lambda(T_{w}(x)-x)+e-z\|\leq\|x-z\|+\eta\sum_{i\in I_{z}}w(i).
\end{equation}
\end{lemma}

\begin{proof} 
Define for $I_{z}\neq\emptyset$

\begin{equation}
\eta:=\sup\{\|x+\lambda(T_{i}(x)-x)+\widetilde{e}-z\|\mid x\in G,\lambda\in[0,2],i\in I_{z},\widetilde{e}\in E_{1}(x,z,\lambda,i)\},\label{eq:eta1}
\end{equation}

and for $I_{z}=\emptyset$ define

\begin{equation}
\eta:=0.\label{eq:eta2}
\end{equation}

If \textbf{(\ref{eq:eta2})} holds, then obviously $\eta\in[0,\infty)$.
Next we show that $\eta\in[0,\infty)$ also when \textbf{(\ref{eq:eta1})}
holds, from which it will follow that \textbf{(\ref{eq:Sum-eta})}
holds regardless if $I_{z}\neq\emptyset$ or not.

Since $G$ is a compact set and since for all $i\in I$ the real function
$f_{i}(x):=2\|T_{i}(x)-x\|$ is continuous on $G$ as a result of
the continuity of the norm and the assumption on $T_{i}$, it follows
from the the well-known Weierstrass Theorem  (that is, the Extreme Value Theorem of calculus) that $f_{i}$ is bounded
from above on $G$. Let $\mu_{i}>0$ be any such upper bound. Elementary
algebra shows that any $\widetilde{e}$ which satisfies \textbf{(\ref{eq:abs(e)})}
(with $\theta:=1$ and $\widetilde{e}$ instead of $e$) also satisfies
$\|\widetilde{e}\|\leq f_{i}(x)$, and hence $\|\widetilde{e}\|\leq\mu_{i}$
whenever $(x,\lambda)\in G\times[0,2]$ and $\widetilde{e}\in E_{1}(x,z,\lambda,i)$.

Since 

\begin{equation}\label{eq:g_i}
g_{i}(x,\lambda,\widetilde{e}):=\|x+\lambda(T_{i}(x)-x)+\widetilde{e}-z\|
\end{equation}

is continuous on $G\times[0,2]\times B[0,\mu_{i}]$ for all $i\in I$
as a result of the continuity of the norm and the assumption on $T_{i}$, the Weierstrass Theorem ensures that $g_{i}$ is bounded from above on $G\times[0,2]\times B[0,\mu_{i}]$. Since $I$ and hence $I_{z}$
are finite, we conclude from the previous assertions and the definition
of $\eta$ that

\begin{equation}
0\leq\eta\leq\max_{i\in I_{z}}\sup\{g_{i}(x,\lambda,\widetilde{e})\mid(x,\lambda,\widetilde{e})\in G\times[0,2]\times B[0,\mu_{i}]\}<\infty,
\end{equation}

and so $\eta\in[0,\infty)$.

Now let $x\in G$, $\lambda\in[0,2]$ be arbitrary, let $w:I\to[0,1]$
be an arbitrary weight function and let $e\in E_{1}(x,z,\lambda,w)$
be arbitrary. Since for all $i\notin I_{z}$ one has $z\in Q_{i}$,
it follows from Lemma \textbf{\ref{lem:PerturbedFejer}} (with $T:=T_{i}$,
$S:=Q_{i}$, $q:=z$ and $e^{i}$ instead of $e$) that

\begin{equation}
\|x+\lambda(T_{i}(x)-x)+e^{i}-z\|\leq\|x-z\|.
\end{equation}

These inequalities and the triangle inequality, together with the
definition of $\eta$, as well as the convention that a sum over the
empty set is zero, ensure that 

\begin{multline}\label{eq:formula}
\left\|x+\lambda(T_{w}(x)-x)+e-z\right\|\\
=\left\|\sum_{i\in I_{z}}w(i)\bigl(x+\lambda(T_{i}(x)-x)+e^{i}-z\bigr)+\sum_{i\notin I_{z}}w(i)\bigl(x+\lambda(T_{i}(x)-x)+e^{i}-z\bigr)\right\| \\
\leq\sum_{i\in I_{z}}w(i)\left\|x+\lambda(T_{i}(x)-x)+e^{i}-z\right\|+\sum_{i\notin I_{z}}w(i)\left\|x+\lambda(T_{i}(x)-x)+e^{i}-z\right\|\\
\leq\left(\sum_{i\in I_{z}}w(i)\right)\eta+\left(1-\sum_{i\in I_{z}}w(i)\right)\|x-z\|\\
=\|x-z\|+\left(\sum_{i\in I_{z}}w(i)\right)(\eta-\|x-z\|)\leq\|x-z\|+\eta\sum_{i\in I_{z}}w(i).
\end{multline}
\end{proof}

\begin{lemma} \label{lem:beta} Let $q\in Q$ be fixed and suppose
that $C\subseteq X$ is a nonempty and compact subset. Denote $I_{C}:=\{i\in I\mid C\cap Q_{i}=\emptyset\}$.
Then there exists $\beta>0$ such that for all $x\in C$, all $\lambda\in[\tau_{1},2-\tau_{2}]$,
all weight functions $w:I\to[0,1]$ and all $e\in E_{0.5}(x,q,\lambda,w)$,
\begin{equation}
\|x+\lambda(T_{w}(x)-x)+e-q\|\leq\|x-q\|-\beta\sum_{i\in I_{C}}w(i).\label{eq:sumI_C}
\end{equation}
\end{lemma}

\begin{proof} If $I_{C}=\emptyset$ then \textbf{(\ref{eq:sumI_C})}
holds with $\beta=1$ which is positive. Hence, from now on we assume
that $I_{C}\neq\emptyset$. For each $i\in I_{C}$ denote 
\begin{multline}\label{eq:beta_i_def}
\beta_{i}:=\inf\{\|x-q\|-\|x+\lambda(T_{i}(x)-x)+\widetilde{e}-q\|\mid\\
x\in C,\lambda\in[\tau_{1},2-\tau_{2}],\widetilde{e}\in E_{0.5}(x,q,\lambda,i)\},
\end{multline}
and let 
\begin{equation}
\beta:=\min\{\beta_{i}\mid i\in I_{C}\}.
\end{equation}
We show next that $\beta>0$, and that \textbf{(\ref{eq:sumI_C})}
holds regardless whether $I_{C}=\emptyset$ or not.

The definition of $\beta_{i}$ and Lemma \textbf{\ref{lem:PerturbedFejer}}
imply that $\beta_{i}\in[0,\infty)$ for every $i\in I_{C}$. Since
$I$ and hence $I_{C}$ are finite, it suffices to show that $\beta_{i}>0$
for all $i\in I_{C}$ in order to conclude that $\beta>0$. Given
$i\in I_{C}$, the definition of $\beta_{i}$ implies that for each
$\ell\in\mathbb{N}$ there exists a triplet $(x_{\ell,i},\lambda_{\ell,i},\widetilde{e}_{\ell,i})\in C\times[\tau_{1},2-\tau_{2}]\times E_{0.5}(x_{\ell,i},q,\lambda_{\ell,i},i)$
such that 
\begin{equation}
\beta_{i}\leq\|x_{\ell,i}-q\|-\|x_{\ell,i}+\lambda_{\ell,i}(T_{i}(x_{\ell,i})-x_{\ell,i})+\widetilde{e}_{\ell,i}-q\|<\beta_{i}+\frac{1}{\ell}.\label{eq:beta_i_sandwitch}
\end{equation}

Because of the compactness of $C\times[\tau_{1},2-\tau_{2}]$, there
exists an infinite set $N_{1}$ of natural numbers, and a pair $(x(i),\lambda(i))\in C\times[\tau_{1},2-\tau_{2}]$,
such that

\begin{equation}
(x(i),\lambda(i))=\lim_{\ell\to\infty,\ell\in N_{1}}(x_{\ell,i},\lambda_{\ell,i}).
\end{equation}

Let

\begin{equation}
\Lambda_{i}:=\inf\{\|\widetilde{e}_{\ell,i}\|\mid\ell\in N_{1}\}
\end{equation}

{\noindent} and let $N_{2}$ be an infinite subset of $N_{1}$ such that the subsequence
$(\|\widetilde{e}_{\ell,i}\|)_{\ell\in N_{2}}$ converges to $\Lambda_{i}$.

We claim that the subsequence $(\widetilde{e}_{\ell,i})_{\ell\in N_{2}}$
has a convergent subsequence. Indeed, since $C$ is a compact set
and since for all $i\in I$ the real function $f_{i}(x):=\|T_{i}(x)-x\|$
is continuous on $C$ as a result of the continuity of the norm and
the assumption on $T_{i}$, it follows from the Weierstrass Theorem 
that $f_{i}$ is bounded from above on $C$. Let $\mu_{i}>0$ be an
arbitrary upper bound on $f_{i}$ over $C$.

Elementary algebra shows that any $\widetilde{e}$ which satisfies
\textbf{(\ref{eq:abs(e)})} (with $\theta:=0.5$ and with $\widetilde{e}$
instead of $e$) also satisfies $\|\widetilde{e}\|\leq f_{i}(x)$,
and hence $\|\widetilde{e}\|\leq\mu_{i}$ whenever $(x,\lambda,\widetilde{e})\in C\times[\tau_{1},2-\tau_{2}]\times E_{0.5}(x,q,\lambda,i)$.

Since $\widetilde{e}_{\ell,i}\in E_{0.5}(x_{\ell,i},q,\lambda_{\ell,i},i)$
for all $\ell\in\mathbb{N}$, it follows that $\|\widetilde{e}_{\ell,i}\|\leq\mu_{i}$
for all $\ell\in\mathbb{N}$ and, in particular, for all $\ell\in N_{2}$.
Because the ball $B[0,\mu_{i}]$ is compact and the sequence $(\widetilde{e}_{\ell,i})_{\ell\in N_{2}}$
is contained in this ball, it indeed has a convergent subsequence
which converges to some vector $\widetilde{e}(i)$ which belongs to
this ball, namely, there exists an infinite subset $N_{3}$ of $N_{2}$
such that $\lim_{\ell\to\infty,\ell\in N_{3}}\widetilde{e}_{\ell,i}=\widetilde{e}(i)$.

Now, when we combine this fact, together with the continuity of the
norm, the fact that $N_{3}\subseteq N_{2}$ and the definition of
$\Lambda_{i}$, we obtain

\begin{equation}
\Lambda_{i}=\lim_{\ell\to\infty,\ell\in N_{3}}\|\widetilde{e}_{\ell,i}\|=\|\widetilde{e}(i)\|.
\end{equation}

In addition, denote by $h_{i}(x,\lambda)$ the function on the right-hand
side of \textbf{(\ref{eq:abs(e)})}, with $\theta:=1/2$ and $h_{i}$
defined on $C\times[\tau_{1},2-\tau_{2}]$. Then $\|\widetilde{e}_{\ell,i}\|\leq h_{i}(x_{\ell,i},\lambda_{\ell,i})$
for all $\ell\in\mathbb{N}$, and, in particular, for all $\ell\in N_{3}$,
because $\widetilde{e}_{\ell,i}\in E_{0.5}(x_{\ell,i},q,\lambda_{\ell,i},i)$
for all $\ell\in\mathbb{N}$.

We recall that the right-hand side of \textbf{(\ref{eq:abs(e)})} vanishes
if its denominator, and hence its numerator, vanish. Thus, it is not
clear that $h_{i}$ is continuous at $(x,\lambda)$ for which the
denominator in the definition of $h_{i}(x,\lambda)$ vanishes, but
from the continuity of the norm and of $T_{i}$ it is clear that $h_{i}$
is continuous at $(x,\lambda)$ whenever the above-mentioned denominator
does not vanish. Anyway, since the definition of $I_{C}$ ensures
that $C\cap Q_{i}=\emptyset$ for all $i\in I_{C}$, and therefore
$x(i)\notin Q_{i}$, namely, $x(i)\neq T_{i}(x(i))$ for all $i\in I_{C}$,
it follows that for all $i\in I_{C}$ the denominator in the definition
of $h_{i}(x,\lambda)$ does not vanish at $(x(i),\lambda(i))$; hence
$h_{i}$ is continuous at $(x(i),\lambda(i))$ for all $i\in I_{C}$.
This fact and the limit $(x(i),\lambda(i))=\lim_{\ell\to\infty,n\in N_{3}}(x_{\ell,i},\lambda_{\ell,i})$,
yield 
\begin{equation}
\|\widetilde{e}(i)\|=\lim_{\ell\to\infty,\ell\in N_{3}}\|\widetilde{e}_{\ell,i}\|\leq\lim_{\ell\to\infty,\ell\in N_{3}}h_{i}(x_{\ell,i},\lambda_{\ell,i})=h_{i}(x(i),\lambda(i)).\label{eq:yield}
\end{equation}
Hence, $\widetilde{e}(i)\in E_{0.5}(x(i),q,\lambda(i),i)$. Thus,
if we let $\theta:=1$ in \textbf{(\ref{eq:abs(e)})}, we see that $\widetilde{e}$
satisfies \textbf{(\ref{eq:abs(e)})} with strict inequality, where in
\textbf{(\ref{eq:abs(e)})} we let $T:=T_{i}$, $\lambda:=\lambda(i)$,
$x:=x(i)$ and $\widetilde{e}(i)$ instead of $e$. Since $0<\tau_{1}\leq\lambda(i)\leq2-\tau_{2}<2$
and $x(i)\neq T_{i}(x(i))$, and since $q\in Q\subseteq Q_{i}$, we
conclude from Lemma \textbf{\ref{lem:PerturbedFejer}} (in which $T:=T_{i}$,
$S:=Q_{i}$, $\widetilde{e}(i)$ is instead of $e$, $\lambda:=\lambda(i)$,
$x:=x(i)$), from \textbf{(\ref{eq:beta_i_sandwitch})}, and from
the continuity of $T_{i}$ and the norm, that for each $i\in I_{C}$,
\begin{multline}
\beta_{i}=\lim_{\ell\to\infty,\ell\in N_{3}}\Big[\|x_{\ell,i}-q\|-\|x_{\ell,i}+\lambda_{\ell,i}(T_{i}(x_{\ell,i})-x_{\ell,i})+\widetilde{e}_{\ell,i}-q\|\Big]\\
=\|x(i)-q\|-\|x(i)+\lambda(i)(T_{i}(x(i))-x(i))+\widetilde{e}(i)-q\|>0.\label{eq:beta-i}
\end{multline}
Finally, since $\beta=\min\{\beta_{i},i\in I_{C}\}$ and since $I$,
and hence $I_{C}$, are finite, it follows that $\beta=\beta_{j}$
for some $j\in I_{C}$, and so indeed $\beta>0$ also in the case
where $I_{C}\neq\emptyset$, as claimed.

The definition of $\beta$ implies that

\begin{equation}
\|x+\lambda(T_{i}(x)-x)+\widetilde{e}-q\|\leq\|x-q\|-\beta
\end{equation}

for all $i\in I_{C}$ and all $(x,\lambda,\widetilde{e})\in C\times[\tau_{1},2-\tau_{2}]\times E_{0.5}(x,q,\lambda,i)$,
with an empty inequality when $I_{C}=\emptyset$. In addition, since
$q\in Q\subseteq Q_{i}$ for each $i\in I$, we infer from Lemma \textbf{\ref{lem:PerturbedFejer}}
that

\begin{equation}
\|x+\lambda(T_{i}(x)-x)+\widetilde{e}-q\|\leq\|x-q\|
\end{equation}

for every $i\in I$ and all $(x,\lambda,\widetilde{e})\in C\times[\tau_{1},2-\tau_{2}]\times E_{0.5}(x,q,\lambda,i)$,
and, in particular, for every $i\notin I_{C}$ (with an empty inequality
if $I\backslash I_{C}=\emptyset$).

It follows from these inequalities and the triangle inequality, together
with the convention that the sum over the empty set is zero, that
for all $x\in C$, all $\lambda\in[\tau_{1},2-\tau_{2}]$, all weight
functions $w:I\to[0,1]$ and all $e\in E_{0.5}(x,q,\lambda,w)$, 
\begin{multline}
\|x+\lambda(T_{w}(x)-x)+e-q\|=\left\| \sum_{i\in I}w(i)\big(x+\lambda(T_{i}(x)-x)+e^{i}-q\big)\right\| \\
=\left\| \sum_{i\in I_{C}}w(i)\left(x+\lambda(T_{i}(x)-x)+e^{i}-q\right)+\sum_{i\notin I_{C}}w(i)\left(x+\lambda(T_{i}(x)-x)+e^{i}-q\right)\right\| \\
\leq\left(\sum_{i\in I_{C}}w(i)\right)(\|x-q\|-\beta)+\left(1-\sum_{i\in I_{C}}w(i)\right)\|x-q\|\\
=\|x-q\|-\beta\sum_{i\in I_{C}}w(i).\label{eq:final}
\end{multline}
Consequently, we established \textbf{(\ref{eq:sumI_C})}, as required.
\end{proof}

The next lemma establishes the Fej\'er monotonicity of sequences generated
by Algorithm \textbf{\ref{alg:GBIP}.}

\begin{lemma} \label{lem:Fejer-kQ} For each $q\in Q\cap B[x^{0},2\sigma]$,
$k\in\mathbb{N}\cup\{0\}$ and each $e^{k}$ which satisfies \textbf{(\ref{eq:abs(eki)})},
one has 
\begin{equation}
\|x^{k+1}-q\|\leq\|x^{k}-q\|,\label{eq:Fejer-kQ}
\end{equation}
and this inequality is strict if there exists some $i\in I$ such
that both $x^{k}\notin Q_{i}$ and $w_{k}(i)>0$. 
\end{lemma}

\begin{proof} We prove the assertion using induction on $k$. First
observe that by the choice of $\sigma$ in Algorithm \textbf{\ref{alg:GBIP}},
one has $d(x^{0},Q)<\sigma$. Therefore, there is some $\widetilde{q}\in Q$
such that $\|x^{0}-\widetilde{q}\|<\sigma$. Hence, $Q\cap B[x^{0},\sigma]\neq\emptyset$
and, since $B[x^{0},\sigma]\subseteq B[x^{0},2\sigma]$, also $Q\cap B[x^{0},2\sigma]\neq\emptyset$.

Now let $q\in Q\cap B[x^{0},2\sigma]$ be arbitrary. Suppose that
$k=0$. Since $\|x^{0}-q\|\leq2\sigma$ and $e^{0,i}$ satisfies \textbf{(\ref{eq:abs(eki)})}
for each $i\in I$, it follows that $e^{0,i}$ satisfies \textbf{(\ref{eq:abs(e)})}
(with $e^{0,i}$ instead of $e$ and with $\theta:=1/2$, and hence also with $\theta:=1$) for each $i\in I$.

This fact, the assumption that $q\in Q\subseteq Q_{i}$ for each $i\in I$,
and the notation $y^{0,i}:=x^{0}+\lambda_{0}(T_{i}(x^{0})-x^{0})+e^{0,i}$
imply, using Lemma \textbf{\ref{lem:PerturbedFejer}} (in which $x:=x^{0}$,
$S:=Q_{i}$, $T:=T_{i}$, $\lambda:=\lambda_{0}$, $e:=e^{0,i}$),
that 
\begin{equation}\label{eq:Fejer-i0}
\|y^{0,i}-q\|\leq\|x^{0}-q\|.
\end{equation}
Since $0<\tau_{1}\leq\lambda_{0}$ and $2-\lambda_{0}\geq\tau_{2}>0$,
if $x^{0}\notin Q_{i}$ then \textbf{(\ref{eq:Fejer-i0})} is strict,
again from Lemma \textbf{\ref{lem:PerturbedFejer}}. These considerations,
\textbf{(\ref{eq:GBIP})} and the triangle inequality imply that 
\begin{multline}\label{eq:fejer}
\|x^{1}-q\|=\|x^{0}+\lambda_{0}(T_{w_{0}}(x^{0})-x^{0})+e^{0}-q\|\\
=\left\| \sum_{i\in I}w_{0}(i)\left(x^{0}+\lambda_{0}(T_{i}(x^{0})-x^{0})+e^{0,i}-q\right)\right\| =\left\| \sum_{i\in I}w_{0}(i)\left(y^{0,i}-q\right)\right\| \\
\leq\sum_{i\in I}w_{0}(i)\left\| y^{0,i}-q\right\| \leq\sum_{i\in I}w_{0}(i)\|x^{0}-q\|=\|x^{0}-q\|,
\end{multline}
and this inequality is strict if there exists some $i\in I$ such
that $x^{0}\notin Q_{i}$ and $w_{0}(i)>0$. In other words, \textbf{(\ref{eq:Fejer-kQ})}
holds true for the case $k=0$.

Suppose now that the assertion holds for all nonnegative integers
up to $k\in\mathbb{N}\cup\{0\}$. We want to show that it holds for
$k+1$ as well. The induction hypothesis implies that $\|x^{k}-q\|\leq\ldots\leq\|x^{1}-q\|\leq\|x^{0}-q\|\leq2\sigma$.
Since $e^{k,i}$ satisfies \textbf{(\ref{eq:abs(eki)})} for each $i\in I$,
it follows that $e^{k,i}$ satisfies \textbf{(\ref{eq:abs(e)})} (with
$e^{k,i}$ instead of $e$ and with $\theta:=1/2$, and hence also
with $\theta:=1$) for each $i\in I$. This fact, the assumption that
$q\in Q\subseteq Q_{i}$ for each $i\in I$, and the notation $y^{k,i}:=x^{k}+\lambda_{k}(T_{i}(x^{k})-x^{k})+e^{k,i}$
imply, using Lemma \textbf{\ref{lem:PerturbedFejer}} (in which $x:=x^{k}$,
$S:=Q_{i}$, $T:=T_{i}$, $\lambda:=\lambda_{k}$, $e:=e^{k,i}$),
that 
\begin{equation}\label{eq:Fejer-ik}
\|y^{k,i}-q\|\leq\|x^{k}-q\|.
\end{equation}
Since $0<\tau_{1}\leq\lambda_{k}$ and $2-\lambda_{k}\geq\tau_{2}>0$,
if $x^{k}\notin Q_{i}$ then \textbf{(\ref{eq:Fejer-ik})} is strict,
again from Lemma \textbf{\ref{lem:PerturbedFejer}}. These considerations,
\textbf{(\ref{eq:GBIP})} and the triangle inequality imply that 
\begin{multline}
\|x^{k+1}-q\|=\|x^{k}+\lambda_{k}(T_{w_{k}}(x^{k})-x^{k})+e^{k}-q\|\\
=\left\| \sum_{i\in I}w_{k}(i)\left(x^{k}+\lambda_{k}(T_{i}(x^{k})-x^{k})+e^{k,i}-q\right)\right\| =\left\| \sum_{i\in I}w_{k}(i)\left(y^{k,i}-q\right)\right\| \\
\leq\sum_{i\in I}w_{k}(i)\left\| y^{k,i}-q\right\| \leq\sum_{i\in I}w_{k}(i)\|x^{k}-q\|=\|x^{k}-q\|,\label{eq:fejer2}
\end{multline}
and this inequality is strict if there exists some $i\in I$ such
that both $x^{k}\notin Q_{i}$ and $w_{k}(i)>0$. In other words,
\textbf{(\ref{eq:Fejer-kQ})} holds true also for $k+1$, as required.
\end{proof}\vspace*{0.3cm}

\begin{proof}[{\bf Proof of Theorem \ref{thm:GBIP} }] 
The proof is divided into several steps. In the following steps we fix
some $q\in Q\cap B[x^{0},\sigma]$. This is possible since $Q\cap B[x^{0},\sigma]\neq\emptyset$
by the assumption that $d(x^{0},Q)<\sigma$. \\\vspace*{0.5cm}

\noindent {\bf Step 1: $(x^{k})_{k=0}^{\infty}$ has a convergent
subsequence}: Indeed, Lemma \textbf{\ref{lem:Fejer-kQ}} ensures
that $\|x^{k+1}-q\|\leq\|x^{k}-q\|\leq\ldots\leq\|x^{0}-q\|$ for
all $k\in\mathbb{N}\cup\{0\}$. Hence, $x^{k}$ is in the compact
ball $B[q,\|x^{0}-q\|]$ for all $k\in\mathbb{N}\cup\{0\}$, so that
$(x^{k})_{k=0}^{\infty}$ must have a convergent subsequence. \\\vspace*{0.5cm}

\noindent  {\textbf{Step 2: $(x^{k})_{k=0}^{\infty}$ has at most one accumulation
point}}: Step 1 ensures that $(x^{k})_{k=0}^{\infty}$ has at least
one accumulation point. Assume to the contrary that it has two different
accumulation points $u$ and $u^{\prime}$. Then $\delta:=\|u-u^{\prime}\|>0$.
As explained in Step 1 above, the sequence $(\|x^{k}-q\|)_{k=0}^{\infty}$
is decreasing, and it is bounded from below by 0. Thus, $\lim_{k\to\infty}\|x^{k}-q\|$
exists. Since $u$ is an accumulation point of $(x^{k})_{k=0}^{\infty}$,
it follows from the norm continuity that $\lim_{k\to\infty}\|x^{k}-q\|=\|u-q\|$ and also that, for all $k\in\mathbb{N}\cup\{0\}$, 
\begin{equation}
\|u-q\|\leq\|x^{k}-q\|.\label{eq:abs(u-q)<=abs(xk-q)}
\end{equation}

\vspace{0.2cm}

\noindent {\textbf{Step 2.1: $u\in Q\cap B[x^{0},2\sigma]$}}: We first
show that $u\in B[x^{0},2\sigma]$. Indeed, \textbf{(\ref{eq:abs(u-q)<=abs(xk-q)})}
implies that $\|u-q\|\leq\|x^{0}-q\|$, and since $\|x^{0}-q\|\leq\sigma$
by the choice of $q$, we have

\begin{equation}
\|u-x^{0}\|\leq\|u-q\|+\|q-x^{0}\|\leq2\|x^{0}-q\|\leq2\sigma.
\end{equation}

Hence, $u\in B[x^{0},2\sigma]$.

Now we show that $u\in Q$. Assume to the contrary that $u\notin Q$.
Then $I_{u}:=\{i\in I\mid u\notin Q_{i}\}$ is nonempty and is finite
since $I$ is finite. Since $Q_{i}$ is closed for each $i\in I$,
it follows that $d(u,Q_{i})>0$ for every $i\in I_{u}$. Let $\rho$
be any positive number which satisfies $\rho<\min\{\delta/2,d(u,Q_{i})\mid i\in I_{u}\}$.
Let $C$ be the closed ball with radius $\rho$ and center $u$. The
choice of $\rho$ implies that $C\cap Q_{i}=\emptyset$ for all $i\in I_{u}$.
Let $I_{C}:=\{i\in I\mid C\cap Q_{i}=\emptyset\}$.

From the previous line $I_{u}\subseteq I_{C}$. On the other hand,
it must be that $I_{C}\subseteq I_{u}$ since if $i\in I_{C}\backslash I_{u}$,
then both $u\in Q_{i}$ (since $i\notin I_{u}$) and $u\notin Q_{i}$
(since $u\in C$ and $C$ is disjoint to $Q_{i}$ according to the
definition of $I_{C}$), a contradiction. Hence, $I_{C}=I_{u}$.

Now let $\eta$ be as in Lemma \textbf{\ref{lem:eta}}, where there
$z:=u$, and let $G$ be the closed ball of radius $\|x^{0}-q\|$
around $q$ (or the closed ball with radius $\|x^{0}-q\|+\|q\|$ around
the origin). Let $\beta$ be as in Lemma \textbf{\ref{lem:beta}}.
Define 
\begin{equation}
\varepsilon:=\frac{\rho\beta}{\eta+\beta}.\label{eq:epsilon}
\end{equation}
Since $\varepsilon>0$ and $u$ is an accumulation point of $(x^{k})_{k=0}^{\infty}$,
there exists an index $k\in\mathbb{N}\cup\{0\}$ sufficiently large
such that 
\begin{equation}
\|x^{k}-u\|<\varepsilon.\label{abs(xk-u)<epsilon}
\end{equation}
Since $\varepsilon<\rho$, it follows that $x^{k}\in C$. Since $u^{\prime}$
is also an accumulation point and $\rho<(1/2)\delta=(1/2)\|u-u^{\prime}\|$, the set of all $\widetilde{k}\in\mathbb{N}\cup\{0\}$ such that $\widetilde{k}>k$ 
and $x^{\widetilde{k}}\notin C$ is nonempty (in fact, infinite).
Let $k^{\prime}$ be the smallest element in this set. Then $k^{\prime}>k$,
and any $\widetilde{k}\in[k,k^{\prime}-1]\cap\left(\mathbb{N}\cup\{0\}\right)$
has the property that $x^{\widetilde{k}}\in C$. Hence, we can apply
Lemma \textbf{\ref{lem:beta}} repeatedly with $x:=x^{\widetilde{k}}$
where $\widetilde{k}\in[k,k^{\prime}-1]\cap\left(\mathbb{N}\cup\{0\}\right)$,
and by using \textbf{(\ref{eq:GBIP})}, \textbf{(\ref{eq:abs(u-q)<=abs(xk-q)})},
\textbf{(\ref{abs(xk-u)<epsilon})}, the triangle inequality and the
choice of $k$, it follows that 
\begin{multline}
\|u-q\|\leq\|x^{k^{\prime}}-q\|=\|x^{k^{\prime}-1}+\lambda_{k^{\prime}-1}(T_{w_{k^{\prime}-1}}(x^{k^{\prime}-1})-x^{k^{\prime}-1})+e^{k^{\prime}-1}-q\|\\
\leq\|x^{k^{\prime}-1}-q\|-\beta\sum_{i\in I_{C}}w_{k^{\prime}-1}(i)\leq\cdots\leq\\
\leq\|x^{k}-q\|-\beta\sum_{t=k}^{k^{\prime}-1}\sum_{i\in I_{C}}w_{t}(i)\leq\|x^{k}-u\|+\|u-q\|-\beta\sum_{t=k}^{k^{\prime}-1}\sum_{i\in I_{C}}w_{t}(i)\\
<\varepsilon+\|u-q\|-\beta\sum_{t=k}^{k^{\prime}-1}\sum_{i\in I_{C}}w_{t}(i).\label{eq:step2.1}
\end{multline}
As a result, 
\begin{equation}
\sum_{t=k}^{k^{\prime}-1}\sum_{i\in I_{C}}w_{t}(i)<\frac{\varepsilon}{\beta}.\label{epsilon/beta}
\end{equation}
In addition, according to the proof of Step 1 and the definition of
the ball $G$ (near \textbf{(\ref{eq:epsilon})}), we have $x^{t}\in G$
for all $t\in\mathbb{N}\cup\{0\}$; in particular, $x^{t}\in G$ for
all $t\in\{k,k+1,\ldots,k^{\prime}-1\}$. Thus, we can apply Lemma
\textbf{\ref{lem:eta}} repeatedly with $w:=w_{t}$, $x:=x^{t}$,
$\lambda:=\lambda_{t}$, $t\in\{k,k+1,\ldots,k^{\prime}-1\}$ and
$z:=u$. By using \textbf{(\ref{eq:GBIP})}, \textbf{(\ref{eq:epsilon})},
\textbf{(\ref{abs(xk-u)<epsilon})}, \textbf{(\ref{epsilon/beta})}
and the equality $I_{C}=I_{u}$, it follows that 
\begin{equation}
\|x^{k^{\prime}}-u\|\leq\|x^{k}-u\|+\eta\sum_{t=k}^{k^{\prime}-1}\sum_{i\in I_{C}}w_{t}(i)<\varepsilon+\eta\sum_{t=k}^{k^{\prime}-1}\sum_{i\in I_{C}}w_{t}(i)\leq\varepsilon+\eta\frac{\varepsilon}{\beta}=\rho.
\end{equation}
Hence $x^{k^{\prime}}\in C$, a contradiction to the choice of $k^{\prime}$.
This contradiction shows that our previous assumption that $u\notin Q$
is invalid. Therefore, $u\in Q$.

\vspace*{0.5cm}

\noindent {\textbf{Step 2.2: $(x^{k})_{k=0}^{\infty}$ converges to $u$,
a contradiction}}: So far we have shown that $u\in Q\cap B[x^{0},2\sigma]$
under the assumption that the sequence $(x^{k})_{k=0}^{\infty}$ has
at least two distinct accumulation points $u$ and $u^{\prime}$.
As a result, we can use Lemma \textbf{\ref{lem:Fejer-kQ}} (where
the $q$ there is replaced by $u$) to conclude that $(\|x^{k}-u\|)_{k=0}^{\infty}$
is a decreasing sequence. Since $u$ is an accumulation point of $(x^{k})_{k=0}^{\infty}$,
it follows that $(\|x^{k}-u\|)_{k=0}^{\infty}$ has a subsequence
which converges to 0, and we conclude that $\lim_{k\to\infty}\|x^{k}-u\|=0$,
namely, $(x^{k})_{k=0}^{\infty}$ converges to $u$. This is a contradiction
to the assumption that $(x^{k})_{k=0}^{\infty}$ has two distinct
accumulation points. 

\vspace*{0.5cm}

\noindent {\textbf{Step 2.3: $(x^{k})_{k=0}^{\infty}$ converges:}} The
previous step shows that the assumption that $(x^{k})_{k=0}^{\infty}$
has more than one accumulation point is invalid. Since, according
to Step 1, $(x^{k})_{k=0}^{\infty}$ has at least one accumulation
point, it follows that this sequence has exactly one accumulation
point, namely it converges. Denote by $x^{\infty}$ its limit. \vspace*{0.5cm}

\noindent {\bf Step 3: $x^{\infty}\in Q\cap B[x^{0},2\sigma]$:}
Indeed, from Lemma ~\textbf{\ref{lem:Fejer-kQ}}, the triangle inequality
and the choice of $q$ it follows that for all $k\in\mathbb{N}\cup\{0\}$,
we have $\|x^{\infty}-x^{0}\|\leq\|x^{\infty}-q\|+\|q-x^{0}\|\leq\|x^{\infty}-x^{k}\|+\|x^{k}-q\|+\sigma\leq\|x^{\infty}-x^{k}\|+\|x^{0}-q\|+\sigma\leq\|x^{\infty}-x^{k}\|+2\sigma$.
Since $\lim_{k\to\infty}\|x^{\infty}-x^{k}\|=0$, by letting $k\to\infty$
in the previous inequality, we have $\|x^{\infty}-x^{0}\|\leq2\sigma$,
namely $x^{\infty}\in B[x^{0},2\sigma]$.

Now, if $\widehat{I}=\emptyset$, then $\widehat{Q}=X$ and, therefore,
$x^{\infty}\in\widehat{Q}$. Suppose now that $\widehat{I}\neq\emptyset$
and assume to the contrary that $x^{\infty}\notin\widehat{Q}$. Before
going further with the proof, it is noteworthy to say that we cannot
use the conclusion of Step 2.1 (with $x^{\infty}$ instead of $u$)
since this step was based on the false assumption that $(x^{k})_{k=0}^{\infty}$
has at least two different accumulation points.

Returning to our goal, the definition of $\widehat{Q}$ and the assumption
that $x^{\infty}\notin\widehat{Q}$ imply that there exists an index
$j\in\widehat{I}$ such that $x^{\infty}\notin Q_{j}$. Therefore,
$d(x^{\infty},Q_{j})>0$. As a result, if we denote by $B^{\infty}$
the closed ball of radius $(1/2)d(x^{\infty},Q_{j})$ and center $x^{\infty}$,
and denote $C:=B^{\infty}$ and $I_{C}:=\{i\in I\mid Q_i\cap C=\emptyset\}$,
then we have $j\in I_{C}$.

Since $(x^{k})_{k=0}^{\infty}$ converges to $x^{\infty}$, there
exists an index $\widetilde{k}\in\mathbb{N}\cup\{0\}$ such that $x^{k}\in C$
for all integers $k\geq\widetilde{k}$. Consequently, by fixing some
$k>\widetilde{k}$ and applying Lemma \textbf{\ref{lem:beta}} repeatedly
with $x:=x^{t}$, $\lambda:=\lambda_{t}$, $w:=w_{t}$ and $e:=e^{t}$,
$t\in[\widetilde{k},k-1]\cap(\mathbb{N}\cup\{0\})$, we conclude that
for all $k>\widetilde{k}$, 
\begin{equation}
\|x^{k}-q\|\leq\|x^{k-1}-q\|-\beta\sum_{i\in I_{C}}w_{k-1}(i)\leq\|x^{\widetilde{k}}-q\|-\beta\sum_{t=\widetilde{k}}^{k-1}\sum_{i\in I_{C}}w_{t}(i),\label{eq:abs(xk-q)<=abs(xk_tilde-q)}
\end{equation}
where $\beta>0$ is the number from Lemma \textbf{\ref{lem:beta}}
with respect to the set $C$. Since $j\in I_{C}$, we have $\sum_{t=\widetilde{k}}^{k-1}w_{t}(j)\leq\sum_{t=\widetilde{k}}^{k-1}\sum_{i\in I_{C}}w_{t}(i)$.
As a result of this inequality and inequality \textbf{(\ref{eq:abs(xk-q)<=abs(xk_tilde-q)})},
we obtain that for all $k>\widetilde{k}$, 
\begin{equation}\label{eq:end}
\sum_{t=\widetilde{k}}^{k-1}w_{t}(j)\leq\sum_{t=\widetilde{k}}^{k-1}\sum_{i\in I_{C}}w_{k-1}(i)\leq\frac{1}{\beta}(\|x^{\widetilde{k}}-q\|-\|x^{k}-q\|)\leq\frac{\|x^{\widetilde{k}}-q\|}{\beta}.
\end{equation}
By letting $k\to\infty$ we conclude that $\sum_{t=\widetilde{k}}^{\infty}w_{t}(j)\leq(1/\beta)\|x^{\widetilde{k}}-q\|<\infty$.
This is a contradiction since $j\in\widehat{I}$ and hence $\sum_{t=\widetilde{k}}^{\infty}w_{t}(j)=\infty$.
Therefore, the assumption $x^{\infty}\notin\widehat{Q}$ cannot hold,
namely, $x^{\infty}\in\widehat{Q}$, as required.
\end{proof}

\begin{remark}\label{rem:Q=00003DQhat} A simple condition which
ensures that a sequence $(x^{k})_{k=0}^{\infty}$ generated by Algorithm
\textbf{\ref{alg:GBIP}} converges to a point located in the common
fixed point set $Q$ is that $Q=\widehat{Q}$; this condition holds
if $I=\widehat{I}$. In other words, we simply need to make sure,
in advance, that $\sum_{k=0}^{\infty}w_{k}(i)=\infty$ for each index
$i\in I$.

This is a rather mild condition. Indeed, it holds in the case of Example
\textbf{\ref{ex:FullySequential}} when the control is repetitive.
It also holds in the case of Example \textbf{\ref{ex:FullySimultaneous}}
when all the weights are equal to $1/m$, or when $w_{k}(i)=1/(mk+m)$
for each $k\in\mathbb{N}\cup\{0\}$ and each $i\in I$, with the exception
of one index $i_{k}\in I$ for which $w_{k}(i_{k})=(mk-1)/(mk+m)$
(there is no restriction at all on $i_{k}$, and yet $\sum_{k=0}^{\infty}w_{k}(i)=\infty$
for all $i\in I$; indeed, fix some $i\in I$ and let $k\geq2$ be
arbitrary; either $i\neq i_{k}$ and then $w_{k}(i)=1/(mk+m)$, or
$i=i_{k}$ and then $w_{k}(i)=(mk-1)/(mk+m)\geq1/(mk+m)$; hence,
$\sum_{k=2}^{\infty}w_{k}(i)\geq\sum_{k=2}^{\infty}(1/(mk+m))=\infty$;
thus, also $\sum_{k=0}^{\infty}w_{k}(i)=\infty$).

Another example is the one given in Example \textbf{\ref{ex:BlockIterative}}
for the control which cycles periodically between the blocks and gives
equal weights to the elements in a specific block. Many more examples
can be given. 
\end{remark}

\begin{remark} \label{rem:fair} One of the assumptions which is
stated in \cite[Algorithm 1, p. 168 and Theorem 1, p. 171]{AharoniCensor1989jour}  
is that the sequence $(w_{k})_{k=0}^{\infty}$ of weight functions
is \textit{fair}, that is, for every $i\in I$ there exist infinitely
many iteration indices $k\in\mathbb{N}\cup\{0\}$ such that $w_{k}(i)>0$.
However, this assumption is never used during the proof of the main
convergence theorem \cite[Theorem 1]{AharoniCensor1989jour}. In our
proof above it is not used as well, and hence we did not even mention
it before the proof. In other words, this assumption is unnecessary.
\end{remark}

\begin{remark}
We want to say a few words regarding possible extensions of this work and the difficulties that one is expected to face when trying to do so.
 
One possible extension is to infinite-dimensional spaces. The main difficulty here is the lack of sequential compactness, as can be seen in: Step 1 in the proof of Theorem \textbf{\ref{thm:GBIP}} (the existence of a convergent subsequence), Step 2 in the proof of Theorem \textbf{\ref{thm:GBIP}} (the existence of accumulation points), the proof of Lemma \textbf{\ref{lem:beta}} (the existence of accumulation points, Weierstrass Theorem) and the proof of Lemma \textbf{\ref{lem:eta}} (Weierstrass Theorem). 

Another possible extension is to cutters which are not necessarily continuous. The difficulty here is mainly in the proofs of Lemma \textbf{\ref{lem:eta}} (Weierstrass Theorem for the functions $g_i$ from \textbf{(\ref{eq:g_i})}) and Lemma \textbf{\ref{lem:beta}} (because of  \textbf{(\ref{eq:beta_i_def})} and \textbf{(\ref{eq:beta-i})}), but the difficulty in Lemma \textbf{\ref{lem:eta}} (and only there) can be overcome if one assumes in advance that each cutter maps bounded sets to bounded sets, since Lemma \textbf{\ref{lem:eta}} is applied (in Step 2.1 of Theorem \textbf{\ref{thm:GBIP}}) to closed balls. 

A third  possible extension is to cutters which are not necessarily defined on the whole space, such as subgradient projections of convex functions which are defined on subsets of the space. Here the whole algorithmic scheme \textbf{(\ref{eq:GBIP})} becomes undefined, but if the subset on which the cutter is defined is closed and convex, then one may overcome the problem (at least for the well-definedness of the scheme) by projecting the right-hand side of \textbf{(\ref{eq:GBIP})} on this subset.  
\end{remark}

\section*{Acknowledgements}

The authors thank the referees for their comments, which helped  improve the paper. The work of the first and second authors was supported by the ISF-NSFC joint research plan, Grant Number 2874/19.

\section*{Data availability} 

Data sharing not applicable to this article as no datasets
were generated or analysed during the current study.

\end{document}